\documentclass[1p]{elsarticle}

\usepackage{lineno,hyperref}

\usepackage{amssymb}
\usepackage{eucal}
\newcommand{\R}{{\Bbb R}}
\newcommand{\Z}{{\Bbb Z}}

\newcommand{\C}{{\Bbb C}}

\usepackage{amsmath}

\newtheorem{thm}{Theorem}
\newtheorem{lemma}[thm]{Lemma}
\newtheorem{corollary}[thm]{Corollary}
\newtheorem{proposition}[thm]{Proposition}

\newtheorem{remark}[thm]{Remark}

\newproof{proof}{Proof}

\begin{document}

\begin{frontmatter}

\title{An iterative estimation for disturbances of semi-wavefronts to the delayed Fisher-KPP equation}

\author[a]{Rafael Benguria D.}
\address[a]{Instituto de F\'isica, Pontificia Universidad Cat\'olica de Chile, Casilla 306, Santiago, Chile
\\ {\rm E-mail: rbenguri@fis.uc.cl}}
\author[b]{Abraham Solar}
\address[b]{Instituto de F\'isica, Pontificia Universidad Cat\'olica de Chile, Casilla 306, Santiago, Chile
\\ {\rm E-mail: asolar@fis.uc.cl}}

\begin{abstract} We give an iterative method to estimate the disturbance of semi-wavefronts of  the equation:  $\dot{u}(t,x) = u''(t,x) +u(t,x)(1-u(t-h,x)),$ $x \in \mathbb{R},\ t >0;$ where $h>0.$ As a consequence, we show the exponential stability, with an unbounded weight, of semi-wavefronts with speed $c>2\sqrt{2}$  and $h>0$. Under the same restriction of  $c$ and $h$, the uniqueness of semi-wavefronts is obtained.      
\end{abstract}


\end{frontmatter}


\section{Introduction}

In this work, we present some results that answer  key questions about bounded solutions of the form $u(t,x):=\phi_c(x+ct)$ to the diffusive Hutchison equation, also called delayed Fisher-KPP equation, with delay $h\geq0$,   
\begin{eqnarray}\label{ie}
\dot{u}(t,x) &=& u''(t,x) + u(t,x)(1-u(t-h,x)),\quad t >0,\ x \in \R,
\end{eqnarray}
(here $\ \dot{} \ $ indicates the temporal derivate whereas $\ ' \ $ indicates the spatial derivate) satisfying  $\phi_c(-\infty)=0$ and $\liminf_{z\to+\infty}\phi_c(z)>0$ for some speed $c>0$ and profile $\phi_c:\R\to\R_+$. Such solutions are called {\it semi-wavefronts} which clearly satisfy the ordinary differential equation with delay

\begin{eqnarray}\label{swe}\nonumber
\phi''_c(z)-c\phi'_c(z)+\phi_c(z)(1-\phi_c(z-ch))=0,\quad z\in\R.
\end{eqnarray}
 If $\phi_c(+\infty)=1$  semi-wavefronts are called {\it wavefronts}.  

If $h=0$, the questions on existence, uniqueness, geometry and stability of wavefronts have been satisfactory responded (see, e.g. \cite{STG} and \cite{UC} and references therein). In this case the general conclusions are: (i) semi-wavefronts are indeed monotone  wavefronts existing for all $c\geq 2$, (ii) two wavefronts with same speed are unique up to translations and (iii) wavefronts are stable under suitable perturbations.

 However, for $h>0$ it has been only recently  established the existence of semi-wavefronts on the domain $\{(h,c)\in\R^2:  h\geq 0\ \hbox{and} \ c\geq 2\}$ (see \cite{HT} and \cite{NP}). The study of the existence of wavefronts to (\ref{ie}) was initiated by Wu and Zou \cite{WZ} who constructed a monotone integral operator for small $h$, depending on $c$, whose iterative application to adequate sub and super-solutions converges to a monotone wavefront. Next, Hasik and Trofimchuk \cite{HT} demonstrated that such monotone wavefronts are unique up to translations. Moreover, the existence and uniqueness of monotone wavefronts has been completely established \cite{AGT,HT}.         
In \cite{HT} and \cite{NP} the authors gave significative information about the geometry of semi-wavefronts, they demonstrated the existence of non-monotone wavefronts in the domain $\{(h,c)\in\R^2: 0\leq h\leq 1\ \hbox{and} \ c\geq 2\}$ and the existence of asymptotically periodic semi-wavefronts whenever $h\geq \pi/2$, in a neighborhood of $\pi/2$, and $c(h)$ sufficiently large (see \cite[Theorem 1.7]{HT} and \cite[Theorem 2.3]{NP}). One of the open problems proposed in \cite{NP} is to prove that the semi-wavefronts are unique, up to translation, in $\{(h,c)\in\R^2:  h\geq 0\ \hbox{and} \ c\geq 2\}.$ In this paper we  obtain the uniqueness of semi-wavefronts in the a region of parameters $(h,c)$ that includes the region $\{(h,c)\in\R^2: h\geq 0\ \hbox{and}\ c>2\sqrt{2}\}$, and we provide a result about the stability of these semi-wavefronts; as much as we know, there are no results concerning the stability of semi-wavefronts of (\ref{ie}) when $h>0$. In our main result, the description of domain the parameters $(h,c)$ to the stability depends on a uniform estimation of semi-wavefronts. Thus, the estimation $\phi_c(z)\leq e^{ch}$, for all $z\in\R$, obtained in \cite{HT},  turns out very profitable.

 The study of the asymptotic behavior of solutions to reaction-diffusion equations with delay normally require some kind of maximum principle. For example, the Mackey-Glass type equations
 \begin{eqnarray}\label{MG}
 \dot{u}(t,x)=u''(t,x)-u(t,x)+g(u(t-h,x)),\quad t>0, x\in\R,
 \end{eqnarray}   
 has been widely studied under the KPP condition $g(u)\leq g'(0)u$, $u>0$, for which, initially, the monotonicity of $g$ was assumed (see, e.g.,  \cite{MOZ}). Then, it was shown in \cite{STR1} that the conclusion of stability of semi-wavefronts  to (\ref{MG}) can be obtained only assuming the Lipschitz continuity of $g$ (type monostable) by a suitable application of the maximum principle. However, this technique does not work with (\ref{ie}) since this equation is strongly non monotone'.       
 
 Our technique is closed to the contractive approach used by Travis and Webb \cite{TW} which  establishes the stability and uniqueness of stationary solutions to nonlinear partial functional differential equation $\dot{u}(t)=Au(t)+F(u_t), t>0$, where $u_t(\theta):=u(t+\theta)$ with $\theta\in[-h,0]$ and $u:[-h,+\infty)\to X$ for some Banach space $X$, $F:C([-h,0],X)\to X$ and $A$ is a certain linear operator. The main assumption in \cite{TW} is that  $A$ generates a strongly continuous semigroup with a growth bound greater than the global Lipschitz constant of $F.$ Unfortunately,  the Travis-Webb approach does not apply to (\ref{ie}) because the nonlinearity is not Lipschitz in the usual spaces. However, we show that this technique can be adapted using a variation of parameter formula given in terms of a certain fundamental solution to the linearization of (\ref{ie}) around of a semi-wavefront. Here we follow the theory of fundamental solutions constructed by Friedman \cite{AF}. 
 
 So, in order to pose the Cauchy problem we take an initial datum in the space  $\mathcal{C}:=C([-h,0],C^{1,1}(\R))$ where, as usual,  
$$
C^{1,1}(\R)=\{f:\R\to\R \ \hbox{is bounded and Lipschitz continuous} \}
$$
  
 \noindent with norm 
  $$|f|_{C^{1,1}}:=\sup_{x\in\R}|f(x)|+\sup_{x\neq y}\frac{|f(x)-f(y)|}{|x-y|}$$ 
and the norm in $\mathcal{C}$ is the sup-norm, i.e.,  $|f|_{\mathcal{C}}:=\sup_{s\in[-h,0]}|f(s)|_{C^{1,1}}$. With all this notation, we give the following global existence result to solutions of (\ref{ie}):   





\begin{proposition}\label{prop1}
If $u_0\in\mathcal{C}$ then there exists a unique solution $u(t,\cdot)$ to (\ref{ie}) such that $u(t,x)=u_0(t,x)$ for all $(t,x)\in[-h,0]\times\R$ and $u_t\in\mathcal{C}$ for all $t\geq 0$. Also, $\lim_{t\to 0}u(t,x)=u_0(0,x)$ for all $x\in\R$ and  if $u_0\geq 0$ then
$u(t,x)\geq 0$ for all $(t,z)\in[-h,+\infty)\times\R$. 
\end{proposition}

\begin{proof}
We denote  the fundamental solution to the heat equation by $\Gamma_0(t,x):=e^{-x^2/4t}/2\sqrt{\pi t}$. Next, from \cite[Lemma 4.1]{UC} the integral equation 
\begin{eqnarray} \label{inte}
u(t,x)=\Gamma_0(t,\cdot)\ast u_0(0,\cdot)(x)+\int_0^t\Gamma_0(t-s,\cdot)\ast[1-u(t-h,\cdot)]u(t,\cdot)(x)ds,
\end{eqnarray}
has a unique bounded solution on $[0,h]\times\R$ which is a solution to (\ref{ie}) since $u(t-h,\cdot)$ is Lipschitz continuous, uniformly with respect to $t\in[0,h]$. Then, differentiating (\ref{inte}) we conclude that
\begin{eqnarray}\label{derivate}
|u_x(t,x)|\leq \operatorname{ess}\sup_{y\in\R}|u_x(0,y)|+\frac{\sqrt{t}}{2\sqrt{\pi}}\sup_{(s,y)\in[0,h]\times\R}|u(s,y)[1-u(s-h,y)]|\ \hbox{on} \ [0,h]\times\R
\end{eqnarray}


\noindent so $u(t,\cdot)$ is  Lipschitz continuous, uniformly respect to $t\in[0,h]$. Repeating the process to the intervals $[h,2h],[2h,3h]...$ the first assertion is true. Next, since the parabolic operator $Lu=u''+(1-u_0)u-\dot{u}$ has a fundamental solution on $[0,h]\times\R$( see, e.g., \cite[Chapter 1, Theorem 10]{AF}) we have that $u(t,x)$ tends to $u_0(0,x)$ when $t$ tends to 0, for all $x\in\R.$ Finally, by the Phragm\`en-Lindel\"of principle \cite[Chapter 3, Theorem 10]{PW} applied to the operator $L$ we obtain $u(t,x)\geq 0$, for all $(t,x)\in[0,h]\times\R$, whenever $u_0\geq 0$ and in a similar form we have $u(t,\cdot)\geq 0$ for the intervals $[h,2h],[2h,3h]...$ 
\end{proof}
\begin{remark}
We note that  if the initial datum is such that $u_0(s,x)=O(x^2)$ the comparison principle works too, however, the boundedness of derivate in (\ref{derivate}) requires the boundedness of $u_0(s, \cdot)$ uniformly in $s\in[0,h]$ and without the uniform boundedness  in time the uniqueness of solution fails  to the step $[0,h]$ (see, e.g.,  \cite{ChKm}). Also, we note that the typical results  to global existence (see, e.g., \cite[Proposition 2.1]{TW} and \cite[Theorem 2.3]{Wu}) do not apply to (\ref{ie}).
\end{remark}

An interesting open problem is to determine whether the solutions of (\ref{ie}), with initial data in $\mathcal{C}$ for example, are globally bounded in time. As much as we know, this issue has been only dealt to the boundary-initial Cauchy problem on a bounded domain (see, e.g., \cite{Fri}).

\section{Main Result}
Now, we denote by $0<\lambda_1(c)\leq \lambda_2(c)$ the zeros of quadratic function 
$$
\epsilon_{\lambda}:=\lambda^2-c\lambda+1.
$$ 
 
Next, for $\lambda\in\R$ we define 
$$
|f|_{\lambda}=\sup_{z\in\R}e^{-\lambda z}|f(z)|=||e^{-\lambda \cdot}f(\cdot)||_{\infty},
$$

and
$$
L_{\lambda}^{\infty}=\{f:\R\to\R, |f|_{\lambda}<\infty \}.
$$


\noindent Also, denoting  by $z=x+ct$ and $v(t,z):=u(t,z-ct)$ we have

\begin{eqnarray}\label{eqc}
\dot{v}(t,z)=v''(t,z)-cv'(t,z)+v(t,z)(1-v(t-h,z-ch)), \quad (t,z)\in\R_+\times\R
\end{eqnarray}

\noindent so, a semi-wavefront $\phi_c$ is a stationary solution of (\ref{eqc}) and we can establish the following result on exponential stability of semi-wavefronts.   

\begin{thm}[Stability with unbounded weigth]\label{sthm}
Let $\lambda_c\in(\lambda_1(c),\lambda_2(c))$. Suppose that $\phi_c$ is semi-wavefront  to (\ref{ie}) and $v_0\in\mathcal{C}$ an initial datum to (\ref{eqc}). Then the following assertions are true 

\begin{itemize} 
\item[(i)] If $v_0$ is non negative such that $v_0(0,\cdot)\in L^{\infty}_{\lambda_c}$ then

\begin{eqnarray}\label{zeroc}\nonumber
\lim_{t\to\infty}v(t,z)=0\quad \hbox{for all} \quad z\in\R.
\end{eqnarray}

\item[(ii)] If 
\begin{eqnarray}\label{norm}\nonumber
e^{-\lambda_c ch}|\phi_c|_{0}\leq \epsilon_{\lambda_c}
\end{eqnarray}
and $\delta\leq 0$ is number satisfying 
\begin {eqnarray}\label{inq1}
e^{-\lambda_c ch}|\phi_c|_{0}\leq e^{\delta h}(\delta+\epsilon_{\lambda_c})
\end{eqnarray}
then 
\begin{eqnarray}\label{inp}
\phi_c(\cdot)-v_0\in C([-h,0],L^{\infty}_{\lambda_c})
\end{eqnarray}
implies
\begin{eqnarray}\label{expst}
|v(t,\cdot)-\phi_c|_{\lambda_c}\leq K_{\lambda_c} e^{\delta t}\quad \hbox{for all}\quad t\geq -h,
\end{eqnarray}
where $K_{\lambda_c}=\max_{s\in[-h,0]}|\phi_c-v_0(s,\cdot)|_{\lambda_c}.$

Finally, if we take $\lambda_c=c/2$ in (\ref{inp}) we obtain (\ref{expst}) on the domain $\{(h,c)\in\R^2: h\geq 0\ \hbox{and}\ c>2\sqrt{2}\}.$ 
\end{itemize}
\end{thm}

Denoting by 
$$
\kappa_h^c:=\frac{c}{2}-\frac{1}{ch}+\frac{1}{2}\sqrt{c^2+\frac{4}{c^2h^2}-4},
$$
the number $\kappa_h^c$ is a critical point of the function $Q:[\lambda_1(c),\lambda_2(c)]\to\R$ defined by
$$
Q(\lambda):=e^{\lambda ch}(-\lambda^2+c\lambda-1),
$$
we can establish the following result for the uniqueness of semi-wavefronts to (\ref{ie}) up to translations.
\begin{corollary}[Uniqueness of semi-wavefronts]\label{coro1}
Let $c>2$. If $\phi_c$ and $\psi_c$ are semi-wavefronts such that
\begin{eqnarray}\label{uniq}
\min\{|\psi_c|_{0},|\phi_c|_{0}\}<Q(\kappa_h^c)
\end{eqnarray}
then $\phi_c(\cdot)=\psi_c(\cdot+z_0)$ for some $z_0\in\R$. In particular, if $c>2\sqrt{2}$ the semi-wavefronts to (\ref{ie}) with speed $c$ are unique up to translations for all $h\geq0$.
\end{corollary}
 \bigskip



\bigskip

Now, for a function $v:[-h,+\infty)\times\R\to\R$ and an initial datum $\eta_0$ we define the following Cauchy problem 
\begin{eqnarray}\label{ACP}
\dot{\eta}(t,z)=\eta''(t,z)-c\eta'(t,z)+[1-v(t-h,z-ch)]\eta(t,z)\quad (t,z)\in\R_+\times\R,\\
\eta(0,z)=\eta_0(z) \ \quad\quad\quad\quad\quad\quad\quad\quad\quad\quad\quad\quad\quad\quad\quad\quad z\in\R.\label{ACP1}
\end{eqnarray}


\begin{lemma}\label{lem1}
Let $\lambda\in\R$ and $v_t\in\mathcal{C}$ for all $t\geq 0$ be. If   $\eta_0\in L^{\infty}_{\lambda}$  then (\ref{ACP})-(\ref{ACP1}) has a unique solution $\eta(t,\cdot)$ in $L^{\infty}_{\lambda}$ for all $t\geq 0$. Moreover, denoting by $\Gamma_v(t)\eta_0:=\eta(t,\cdot)$  we have the following estimate  
\begin{eqnarray}\label{estimate1}
|\Gamma_v(t)\eta_0|_{\lambda}\leq e^{-\epsilon_{\lambda}t}|\eta_0|_{\lambda}\quad\hbox{for all}\quad t\geq 0.
\end{eqnarray} 
\end{lemma}
\begin{proof}
Making the change of variable $p(t,z):=e^{-\lambda z}\eta(t,z)$ we have that $p$ satisfies  the Cauchy problem
\begin{eqnarray}\label{ACP2}
\dot{p}(t,z)=p''(t,z)+(2\lambda-c)p'(t,z)+(\epsilon_{\lambda}-v(t-h,z-ch))p(t,z)\quad(t,z)\in\R_+\times\R\\
 p(0,z)=e^{-\lambda z}\eta_0(z).\quad\quad\quad\quad z\in\R \label{ACP3}
\end{eqnarray}

\noindent If we translate the spatial variable by $(2\lambda-c)t \ $  then, using the Lipschitz continuity of the function $\tilde{v}(t,\cdot):=v(t,\cdot-(2\lambda-c)t)$ (uniformly in $t\in[0,T]$, for any $T<\infty$),  by \cite[Lemma 4.1]{UC} we conclude that (\ref{ACP2})-(\ref{ACP3}) has a unique bounded solution (uniformly in $t\in[0,T]$, for any $T<\infty$), so $\eta(t,\cdot)\in L^{\infty}_{\lambda}$ for all $t\geq 0$. Then, by the Phragm\`en-Lindel\"of principle $\eta(t,\cdot)\geq 0$ for all $t\geq 0$ and the estimation (\ref{estimate1}) is obtained from \cite[Lemma 4.1]{UC}.

\end{proof}

\begin{proof}[Theorem \ref{sthm}]
(i)   It follows from Proposition \ref{prop1} and Lemma \ref{lem1} with $\eta=v$ and $\lambda=\lambda_c$.



\noindent (ii) We denote by $U=[0,h]\times\R$, $R=|\phi_c|_{0}e^{-\lambda_c ch}$ and $w_k(t,x):= w(t+(k-1)h,x)$  where $k=0,1,2,...$
If we take $w=v-\phi_c$ then
\begin{eqnarray}\label{eqt}
\dot{w}_k(t,z)=w''_k(t,z)-cw'_k(t,z)+[1-v_{k-1}(t,z-ch)]w_k(t,z)+f_{k-1}(t,z), \ \hbox{on} \ U,
\end{eqnarray}

\noindent where $f_{k-1}(t,\cdot)=-\phi_c(\cdot) \ w_{k-1}(t,\cdot-ch).$


Now, we define the sequence $C_k:=K_{\lambda_c}e^{(k-1)\delta h}$ for $k=0,1,2,...$ 

\noindent Clearly, $|w_0(t)|_{\lambda_c}\leq C_0e^{\delta t}$ for $t\in[0,h]$. Next, we assume that  

\begin{eqnarray}\label{ih}
|w_{l}(t)|_{\lambda_c}\leq C_{l}e^{\delta t}\quad \hbox{for all}\quad t\in[0,h]
\end{eqnarray}

 \noindent and some  $l\geq1.$ 
 
 \noindent Next, by Lemma \ref{lem1} there exist $D_l>0$ such that
  \begin{eqnarray}\label{idk}
 |f_l(t,z)|\leq D_le^{cz/2}, \ \hbox{on} \ U.
 \end{eqnarray}
Therefore, using Proposition \ref{prop1},   \cite[Chapter 1,Theorem 12]{AF} implies that the associated Cauchy problem  to (\ref{eqt}) with initial datum $w_{l+1}(0,\cdot)$ has a unique bounded solution which is represented by
\begin{eqnarray}\nonumber
w_{l+1}(t)=\Gamma_{v_l}(t) w_{l+1}(0)+\int_0^t \Gamma_{v_l}(t-\tau) f_{l}(\tau)d\tau.
\end{eqnarray}

\noindent Next, using $w_{l+1}(0)=w_l(h)$ and Lemma \ref{lem1} we have

\begin{eqnarray}\nonumber
|w_{l+1}(t)|_{\lambda_c}&\leq& e^{-\epsilon_{\lambda_c}t}|w_{l}(h)|_{\lambda_c}+ R\int_0^t e^{-\epsilon_{\lambda_c}(t-\tau)}|w_{l}(\tau)|_{\lambda_c}d\tau.
\end{eqnarray}
Because of (\ref{ih}) we have
\begin{eqnarray}\nonumber
|w_{l+1}(t)|_{\lambda_c}&\leq& e^{-\epsilon_{\lambda_c}t}e^{\delta h}C_{l}+Re^{-\epsilon_{\lambda_c}t}C_l[e^{(\epsilon_{\lambda_c}+\delta)t}- 1]/(\epsilon_{\lambda_c}+\delta)\\ \nonumber
&=& e^{\delta t}C_{l}[e^{-(\epsilon_{\lambda_c}+\delta)t+\delta h}+R(1-e^{-(\epsilon_{\lambda_c}+\delta)t})/(\epsilon_{\lambda_c}+\delta)].
\end{eqnarray}
Now, due to (\ref{inq1}), the function $M:[0,h]\to\R$ defined by 
$$
M(t)=e^{-(\epsilon_{\lambda_c}+\delta)t+\delta h}+R(1-e^{-(\epsilon_{\lambda_c}+\delta)t})/(\epsilon_{\lambda_c}+\delta)
$$

\noindent is  non increasing. Therefore, 
$$
|w_{l+1}(t)|_{\lambda_c}\leq e^{\delta (t+h)}C_l=e^{\delta t}C_{l+1}\quad\hbox{for all}\quad t\in[0,h].
$$

\noindent So, the general conclusion is
\begin{eqnarray}\nonumber
|w(t,\cdot)|_{\lambda_c}\leq K_{\lambda_c}e^{\delta t} \quad \hbox{for all}\quad t\geq -h.
\end{eqnarray}
 Finally, due to the estimate $|\phi_c|_{0}<e^{ch}$ \cite[Theorem 4.3]{HT}, taking $\lambda_c=c/2$ we have that (\ref{norm}) is satisfied for all $c>2\sqrt{2}$. 
 \end{proof}

\begin{proof}[Corollary \ref{coro1}]
By a standard argument of ordinary differential equations the asymptotic behavior of non negative solutions of (\ref{swe}) at $-\infty$ is
\begin{eqnarray}\label{asw}
\phi_c(z)=A_{\phi_c}z^{j_c}e^{\lambda_1z}+o(z^je^{\lambda_1z})\\ \nonumber
\psi_c(z)=A_{\psi_c}z^{j_c}e^{\lambda_1z}+o(z^je^{\lambda_1z})
\end{eqnarray}
 where $A_{\phi_c}, A_{\psi_c}>0$, $j_c=0,1$  and $j_c=1$ if and only if $c=2$. So, for some $z_0\in\R$ we have $A_{\psi_c(\cdot+z_0)}=A_{\phi_c(\cdot)}$.
 
 Otherwise, denoting $\Phi_c(z)=\phi_c(z)-\psi_c(z+z_0)$ we have
\begin{eqnarray*}\label{equn}
\Phi''_c(z)-c\Phi'_c(z)+(1-\psi_c(z))\Phi_c(z)-\phi_c(z)\Phi_c(z-ch)=0\quad z\in\R.
\end{eqnarray*}
 
 \noindent Next, using (\ref{asw}) and  applying \cite[Proposition 7.2]{MP} we can conclude 
\begin{eqnarray}\nonumber
\Phi_c(z)&=&O(e^{\lambda_2z})\quad \hbox{if}\quad c>2.
\end{eqnarray}

Next, we take, for example, the initial datum $v_0(s,z)=\min\{1,A_{\phi_c}e^{\lambda_1z}\}\in\mathcal{C}$. Then, applying Theorem  \ref{sthm} with $\lambda_c=\kappa_h^c$, $w_0^1=v_0-\phi_c$ and $w_0^2=v_0-\psi_c$ we conclude that $\phi_c(\cdot)=\psi_c(\cdot+z_0)$ under the  condition (\ref{uniq}). Similarly, taking $\lambda_c=c/2$ we conclude   that $\phi_c(\cdot)=\psi_c(\cdot+z_0)$ on the domain $\{(h,c)\in\R^2: h\geq 0\ \hbox{and}\ c>2\sqrt{2}\}$.
\end{proof}

 \section*{Acknowledgments} This work was supported by FONDECYT (Chile) through the Postdoctoral  Fondecyt  2016 program with project number 3160473, and FONDECYT project 116--0856.
 The authors are  grateful to Sergei Trofimchuk for his important comments on this work and to Fran\c cois Hamel for some technical advice.

\end{document}